\documentclass[a4paper, 11pt, reqno,final]{amsart}
\usepackage{amsmath}
\usepackage{amssymb}
\usepackage{upgreek}
\usepackage{graphicx}
\usepackage{color}
\usepackage[T1]{fontenc}
\usepackage[textwidth=440pt, textheight=650pt]{geometry}
\usepackage{enumerate}
\usepackage{hyperref}
\usepackage{showlabels}

\DeclareMathOperator{\N}{\mathbb{N}}
\DeclareMathOperator{\eps}{\varepsilon}

\DeclareMathOperator{\Bbar}{\overline{B}}
\DeclareMathOperator{\Cbar}{\overline{C}}

\DeclareMathOperator{\supp}{supp}

\DeclareMathOperator{\E}{\mathbb{E}}
\DeclareMathOperator{\R}{\mathbb{R}}
\DeclareMathOperator{\Prob}{\mathbb{P}}
\DeclareMathOperator{\ess}{ess}
\DeclareMathOperator{\fS}{\mathfrak{S}}
\DeclareMathOperator{\dimH}{\dim\sb{\mathrm{H}}}
\DeclareMathOperator{\dimB}{\dim\sb{\mathrm{B}}}
\DeclareMathOperator{\dimA}{\dim\sb{\mathrm{A}}}
\DeclareMathOperator{\dimqA}{\dim\sb{\,\mathrm{qA}}}
\DeclareMathOperator{\dimAt}{{\dim}_{\mathrm{A}}^{\,\theta}}
\DeclareMathOperator{\dimAp}{{\dim}_{\mathrm{A}}^{\phi}}

\DeclareMathOperator{\dimAf}{{\dim}_{\mathrm{A}}^{\,\varphi}}
\DeclareMathOperator{\fii}{\varphi}

\newtheorem{theorem}{Theorem}[section]

\newtheorem{corollary}[theorem]{Corollary}
\newtheorem{lemma}[theorem]{Lemma}
\newtheorem{remark}[theorem]{Remark}

\numberwithin{equation}{section}

\title{Assouad spectrum thresholds for some random constructions}
\author{Sascha Troscheit}
\thanks{The author was initially supported by NSERC Grants 2014-03154 and 2016-03719, and the
University of Waterloo.}
\address[Sascha Troscheit]{Faculty of Mathematics\\University of Vienna\\Oskar Morgenstern Platz~1\\1090 Wien\\Austria.}
\email{sascha.troscheit@univie.ac.at}
\urladdr{https://www.mat.univie.ac.at/~troscheit/}
\date{\today}

\begin{document}
\begin{abstract}
  The Assouad dimension of a metric space determines its extremal scaling properties. 
  The derived notion of the Assouad spectrum fixes relative scales by a scaling function to obtain
  interpolation behaviour between the quasi-Assouad and box-counting dimensions. While the quasi-Assouad
  and Assouad dimensions often coincide, they generally differ in random constructions. In this paper
  we consider a generalised Assouad spectrum that interpolates between the quasi-Assouad to the
  Assouad dimension. For common
  models of random fractal sets we obtain a dichotomy of its behaviour by finding a threshold function where
  the quasi-Assouad behaviour transitions to the Assouad dimension. This threshold can be considered
  a phase transition and we compute the threshold for the Gromov boundary of Galton-Watson
  trees and one-variable random
  self-similar and self-affine constructions. We describe how the stochastically self-similar model
  can be derived from the Galton-Watson tree result.
\end{abstract}

\maketitle

\section{Introduction}
The Assouad dimension is an important notion in embedding theory due to the famous Assouad embedding
theorem \cite{assouadphd,Assouad79}, which implies that a metric space $X$ cannot be embedded by a
bi-Lipschitz map into $\R^d$ for any $d$ less than the Assouad dimension of $X$. 
The Assouad dimension is therefore a good indicator of thickness in a metric space and 
is an upper bound to most notions of dimensions
in use today~\cite{Fraser14,RobinsonBook}. In particular, it is an upper bound to the Hausdorff, box-counting, and packing
dimension. Heuristically, the Assouad dimension ``searches'' for the thickest part of a space
relative to two scales $0<r<R$ by finding the minimal exponent $s$ such that the every $R$-ball can be
covered by at most $(R/r)^s$ balls of diameter $r$.

Over the last few years much progress has been made towards our understanding of this dimension and
it
is now a crucial part of fractal geometry, see e.g.\
\cite{Chen19,Fraser14,Fraser14c,Garcia16,Kaenmaki16,Troscheit19} and references therein.
Several other notions of dimension were derived from its definition and this family of
Assouad-type dimensions has attracted much interest. An important notion is the $\theta$-\emph{Assouad spectrum}
introduced by Fraser and Yu \cite{Fraser16a} which aims to interpolate between the upper box-counting and the
Assouad dimension to give fine information on the structure of metric spaces, see
\cite{FraserSurvey} for a recent survey. It analyses sets by
fixing the relation $r = R^{1/\theta}$ in the definition of the Assouad dimension for parameters
$\theta\in(0,1)$.

It turns out that the Assouad spectrum interpolates between the upper box-counting dimension and the
quasi-Assouad dimension introduced by L\"u and Xi~\cite{Lu16}. That is, for $\theta\to 0$ the
$\theta$-Assouad spectrum tends to the upper box-counting dimension, whereas for $\theta\to 1$ it
approaches the quasi-Assouad dimension, see \cite{Spectrum}.
In fact, the quasi-Assouad could be defined in terms of the Assouad
spectrum. 

In many cases the quasi-Assouad dimension and Assouad dimension coincide and the Assouad spectrum
gives best relative scaling information. However, in many
stochastic settings they differ. This can be explained by the Assouad dimension picking up very
extreme behaviour that is almost surely lost over all geometric scales~\cite{Fraser14c}.

In their landmark paper \cite{Fraser16a}, Fraser and Yu discuss the possibility of 
extending the definition of the Assouad spectrum
to analyse the case when quasi-Assouad and Assouad dimension differ. These general spectra, which we
shall also refer to as \emph{generalised Assouad spectra}, would then shed some line on the
behaviour of `in-between' scales. This is done by changing the relation $r=R^{1/\theta}$ to a
general dimension function $r=\varphi(R)$. 
Garc\'{i}a, Hare and Mendivil studied this notion of spectrum (including their natural dual: the lower
Assouad dimension spectrum) and obtain similar interpolation results as are the case for the Assouad
spectrum, see \cite{Garcia19}. 
A common observation is that the intermediate spectrum is constant
and equal to either the quasi-Assouad or Assouad dimension around a \emph{threshold function}. That
is, there exists a function $\phi(x)$ such that $\dimAf F = \dimqA F$ for $\fii(x) =
\upomega(\phi(x))$ and $\dimAf F = \dimA F$ for $\fii(x) = o(\phi(x))$, where we have used the
standard little omega- and $o$-notation\footnote{A function is $f(x) = o(g(x))$ if $f/g\to 0$ as $x\to 0$.
  Similarly, $f(x) = \upomega(g(x))$ if
$g(x) = o(f(x))$.}.
The standard examples where quasi-Assouad and Assouad dimensions differ are random constructions and
this threshold can be considered a phase transition in the underlying stochastic process.
In this paper we will explore this threshold function for various random models.

Garc\'ia et al.~\cite{Garcia19a} considered the following random construction: Let $(l_i)$ be a
non-increasing sequence such that $\sum l_i = 1$. For each $i$, let $U_i$ be an i.i.d.\ copy of $U$,
the random variable that is uniformly distributed in $[0,1]$. Note that, almost surely $U_i \neq
U_j$ for all $i\neq j$. Therefore, almost surely, there is a total ordering of the $(U_i)$.
The complementary set of the random arrangement $E$ is defined as the complement of arranging open
intervals of length $l_i$ in the order induced by $(U_i)$. That is, 
\[
  E = \bigcup_{y\in [0,1]}\left\{ x=\sum_{V_y} l_i \;: \; V_y = \{i : U_i < y\}\right\}.
\]
Almost surely, this set in uncountable and has a Cantor-like structure. Garc\'ia et al.\ have
previously determined the (quasi-)Assouad dimension of deterministic realisations~\cite{Garcia16},
where the order is taken as $U_i < U_{i+1}$ as well as the Cantor arrangement, when $U_i$ is equal
to the right hand end point of the canonical construction intervals of the Cantor middle-third set (ignoring
repeats).
They confirmed in \cite{Garcia19a} that the quasi-Assouad dimension of $E$ is almost surely equal to the quasi-Assouad
dimension of the Cantor arrangement $C$, whereas the Assouad dimension takes the value $1$. They further
gave the threshold function, which they computed as $\phi(x) = \log|\log x|/|\log x|$.
\begin{theorem}[Garc\'ia et al.\ \cite{Garcia19a}]\label{thm:Garcia}
  Let $\phi(x) = \log|\log x|/|\log x|$ and $l_i$ be a decreasing sequence such that $\sum l_i =1$.
  Assume further that there exists  $\eps>0$ such that  
  \[\eps<\frac{\sum_{j\geq 2^{n+1}}l_j}{\sum_{j\geq 2^{n}}l_j}< 1-\eps.\]
  Then, almost surely, $\dimAf E = \dimAf C$ for $\fii=\upomega(\phi(x))$ and $\dimAf E = 1$ for $\fii
  = o(\phi(x))$.
\end{theorem}

In this article we give elementary proofs of the threshold dimension functions for several
canonical random sets.
Under separation conditions we obtain the threshold for one-variable random iterated function
systems with self-similar maps and self-affine maps of Bedford-McMullen type. We also determine the
threshold for the Gromov boundary of Galton-Watson trees. While we do not state it explicitly, using the methods
found in \cite{Troscheit19}, our results for Galton-Watson trees directly applies to stochastically self-similar and
self-conformal sets as well as fractal percolation sets. 

Our proofs rely on the theory of large deviations as well as a
dynamical version of the Borel-Cantelli lemmas.

\section{Definitions and Results}
Let $\phi:\R^+\to\R^+$. We say that $\phi$ is a \emph{dimension function} if $\phi(x)$ and
$\phi(x)|\log x|$
are monotone. Let $N_r(X)$ be the
minimal number of sets of radius at most $r$ needed to cover $X$. The generalised
Assouad spectrum (or intermediate Assouad spectrum) with respect to $\phi$ is given by 
\begin{multline*}
  \dimAp F = \inf\bigg\{s\;:\; (\exists C>0)(\forall 0< r=R^{1+\phi(R)} < R <1) \\\sup_{x\in
  F}N_r(F\cap B(x,R)) \leq C \left(\frac{R}{r}\right)^s = R^{-\phi(R)s}\bigg\}.
\end{multline*}
We will also refer to this quantity as the $\phi$-Assouad dimension of $F$.
Many other variants of the Assouad dimension can now be obtained by restricting $\phi$ in some way:
The Assouad spectrum $\dimAt$ considered by Fraser and Yu 
can be obtained by setting $\phi(R)=1/\theta-1$. The quasi-Assouad dimension $\dimqA F$ is given
by the limit $\dimqA F = \lim_{\theta\to1}\dimAt F$, whereas the Assouad dimension is
obtained by letting $\phi(R)=0$ and allowing $r\leq R$.
We note that we define the generalised Assouad spectrum slightly differently than in
\cite{Fraser16a}.
Instead of requiring $r = \varphi(R)$, we consider the dimension function $\phi$ setting $r=
R^{1+\phi(R)}$. Hence, $\varphi(R)/R = R^{\phi(R)}$. We use this notation as it is slightly more
convenient to use.

\vskip.5em
In fractal geometry two canonical models are used to obtain random fractal sets: stochastically
self-similar sets and one-variable random sets. Instead of computing the result for stochastically
self-similar sets, we compute it for Galton-Watson processes from which the stochastically
self-similar case will follow. 
We will then move on to one-variable random constructions and analyse random self-similar
constructions, as well as a
randomisation of Bedford-McMullen carpets.

\subsection{Galton-Watson trees and stochastically self-similar sets}
Let $X$ be a random variable that takes values in $\{0,\cdots,N\}$ and write $\theta_j =
\Prob\{X = j\}$ for the probability that $X$ takes value $j$.  The Galton-Watson process $Z_i$ is
defined inductively by letting $Z_0=1$ and $Z_{k+1} = \sum_{j=1}^{Z_k} X_j$, where each summand
$X_j$ is an i.i.d.\ copy of $X$. 
The Galton-Watson tree is obtained by considering a tree with single root and determining ancestors
for every node with law $X$, independent of all other nodes. The number of nodes at level $k$ is
then given by $Z_k$ and, conditioned on non-extinction, this process generates a (random) infinite
tree. We endow the set of infinite descending paths starting at the root 
with the standard metric $d(x,y)=2^{-\lvert x\wedge y\rvert}$, where $\lvert x \wedge y\rvert$ is the
level of the least common ancestor. This gives rise to the Gromov boundary of the random
tree that we refer to as $F_\tau$. 

Throughout, we assume that we are in the
supercritical case, i.e.
\[
  m=\E(X)=\sum_{k=1}^N \theta_k k>1.\]
The normalised Galton-Watson process is defined by $W_k = Z_k / m^k$. It is a standard application
of the martingale convergence theorem to show that $W_k\to W$ almost surely. Conditioned on
non-extinction we additionally have $W\in(0,\infty)$ almost surely.
We refer the reader to \cite{Liu00} and \cite{Lyons95} for some other fundamental dimension theoretic
results of Galton-Watson processes.

It was established in \cite{Fraser14c} that the Gromov boundary, using the standard metric, has
Assouad dimension $\log N / \log 2$ almost surely.
In \cite{Troscheit19}, the quasi-Assouad dimension was computed as 
\begin{equation*}
  \dimqA F_\tau = \frac{\log\E(X)}{\log 2} = \dimB F_\tau = \dimH F_\tau
\end{equation*}
almost surely.
This means, in particular, that the $\theta$-Assouad spectrum is constant and equal to the Hausdorff
dimension\footnote{In all models considered in this paper (except the self-affine construction), the Hausdorff and box-counting dimensions
coincide almost surely and all instances of $\dimB$ may be replaced by $\dimH$.}. This is in fact the typical behaviour for all dimension
functions $\phi(x) \leq C (\log|\log x|/|\log x|)$, where $C>0$ is some constant. Conversely, for
$\phi(x) = \upomega
(\log|\log x|)/|\log x|)$ we recover the Assouad dimension, c.f.\ Theorem~\ref{thm:Garcia}. 

\begin{theorem}\label{thm:GWupper}
  Let $F_\tau$ be the Gromov boundary of a supercritical Galton-Watson tree with bounded
  offspring distribution. Let 
  \begin{equation*}
    \phi(x) >  \frac{\log|\log x|}{|\log x|}(\varepsilon\log m)^{-1}
  \end{equation*}
  be a dimension function.
  Then 
  \begin{equation*}
    \dimAp F_\tau \leq (1+\varepsilon) \dimB F_\tau
  \end{equation*}
  almost surely.
\end{theorem}
Note that we trivially have $\dimAp F \geq \dimB F$ and so
\begin{corollary}
   Let $\phi(x) = \upomega(\log|\log x|/|\log x|)$ be a dimension function.
  Then $\dimAp F_\tau =  \dimB F_\tau = \dimH F_\tau$
  almost surely.
\end{corollary}

For the reverse direction we obtain the following bound.
\begin{theorem}
  \label{thm:bottomGW}Let $F_\tau$ be the Gromov boundary of a supercritical Galton-Watson
  tree with bounded
  offspring distribution. Let $N$ be the maximal integer s.t.\ $\theta_N>0$. Further, let
  \begin{equation}
    \phi(x) <  \frac{\log|\log x|}{|\log x|}\frac{\log(N/m)}{\varepsilon \log m \log
    N}.\label{eq:condupper}
  \end{equation}
  be a dimension function. Then, almost surely, 
  \begin{equation*}
    \dimAp F_\tau \geq \min\{(1+\varepsilon) \dimB F_\tau ,\dimA F_\tau\}.
  \end{equation*}
  almost surely.
\end{theorem}
Assume that $m<N$ and so $\dimB F_\tau < \dimA F_\tau$ almost surely.
Then, for $\varepsilon$ satisfying $(1+\varepsilon)=\dimA F / \dimB F=\log
N/m$, we obtain $\dimAp F_\tau \geq \dimA F_\tau$ almost surely. It follows that for 
\[
  \phi(x) < \frac{\log|\log x|}{\log(1/x)}\frac{\log (N/m)}{(N/m)\log m\log N}
  =\frac{m}{N}\left( \frac{1}{\log m}-\frac{1}{\log N}\right) \frac{\log|\log x|}{|\log x|}
\]
we get $\dimAp F = \dimA F$.
\begin{corollary}
   There exists $C>0$ such that, almost surely, $\dimAp F_\tau =  \dimA F_\tau$
 for all dimension functions $\phi(x) \leq C\log|\log x|/|\log x|)$.
\end{corollary}
We will prove both theorems in Section \ref{sect:GWProofs}.
These two bounds are not optimal in the sense that there is a slight gap between the upper and the
lower bound. This gap, after rearranging, is of order
$1-\log m / \log N = 1- \dimB F / \dimA F$.
Let $\phi_\varepsilon$ be equal to the right hand side of \eqref{eq:condupper}. 
We can combine Theorems \ref{thm:GWupper} and \ref{thm:bottomGW} to give
\[
  (1+\varepsilon)\dimB F \leq \dim_{\textrm{A}}^{\phi_\varepsilon} F \leq (1+\varepsilon)\frac{\dimA F}{\dimA
  F - \dimB F}\dimB F
  \]
for an appropriate range of $\varepsilon > 0$.

\subsubsection{Stochastically self-similar sets}
This result can also be applied in the setting of stochastically self-similar sets that were first
studied by Falconer \cite{Falconer86} and Graf \cite{Graf87}. Since we do not exclude the case where
there is no descendant, the analysis also applies to fractal percolation in the sense of Falconer and
Jin~\cite{Falconer15}.
In the case of Mandelbrot percolation, where a $d$-dimensional cube is split into $n^d$ equal
subcubes of sidelength $1/n$ and is kept with probability $p>0$, the number of subcubes is a
Galton-Watson process and the surviving subcubes at level $k$ can be modelled by a Galton-Watson
tree. Since subcubes at the same level have the same diameters, the limit set is almost surely
bi-Lipschitz to the Gromov boundary of an appropriately set up Galton-Watson tree with the small
caveat that the graph metric need to be changed from $d(x,y)=2^{-|x\wedge y|}$ to
$d'(x,y)=n^{|-x\wedge y|}$. This change, however, just affects the results above by a constant and
not their asymptotic behaviour.
For non-homogeneous self-similar sets, where the size may vary at a given generation, one needs to
set up a Galton-Watson tree that models the set. This is described in full detail in
\cite{Troscheit19} and we omit its full derivation for this model here. Using those methods,
Theorems \ref{thm:GWupper} and \ref{thm:bottomGW} become the following corollary.
\begin{corollary}
  Let $F_\tau$ be a stochastically self-similar set arising from finitely many self-similar IFS that
  satisfy the uniform open set condition. Then, if $\phi(x) = \upomega (\log|\log x|/|\log x|)$, we
  obtain $\dimAp F_\tau = \dimB F_\tau$ almost surely.
  Conversely, if $\phi(x) = o(\log|\log x|/|\log x|)$, then $\dimAp F_\tau = \dimA F_\tau$.
\end{corollary}

\subsection{One-variable random sets}
A different popular model for random fractal sets is the one-variable model. It is sometimes also referred to
as a homogeneously random construction. 
We will avoid the latter term to avoid ambiguity with homogeneous iterated function systems.
Let $\Lambda\subset \R^n$ be a compact set supporting the Borel-probability measure $\mu$. For each
$\lambda\in\Lambda$ we associate an iterated function system $\mathbb{I}_\lambda=
\{f^\lambda_1,\dots ,f^\lambda_{N_\lambda}\}$, where each $f_i^\lambda$ is a strictly contracting
diffeomorphism on some non-empty open set $V$.
Throughout this section we make the standing assumption that $\sup_\lambda N_\lambda < \infty$,
that $0<\inf_{\lambda,i,x}\lvert(f_i^\lambda)'(x)\rvert
\leq \sup_{\lambda,i,x}\lvert(f_i^\lambda)'(x)\rvert <1$, and that there exists a non-empty compact set
$\Delta\subset V$ such that
$f_i^\lambda(\Delta)\subseteq \Delta$ for all $\lambda\in\Lambda$ and $1\leq i \leq N_\lambda$. 
To each $\omega\in\Omega=\Lambda^{\N}$, we associate the set $F_\omega$ given by
\begin{equation*}
  F_\omega = \bigcap_{k=1}^\infty\; \bigcup_{\substack{1\leq i_j \leq N_{\lambda_j}\\1\leq j \leq
  k}} f_{i_k}^{\lambda_k}\circ \dots\circ f_{i_1}^{\lambda_1}(\Delta).
\end{equation*}
Let $\Prob=\mu^{(\N)}$ be the product measure on $\Omega=\Lambda^{\N}$. 
We write 
\begin{equation*}
  \E(X(\omega)) = \int_\Omega X(\omega) d\Prob(\omega)
  \quad \text{and} \quad
  \E^g(X(\omega)) =\exp \int_\Omega \log X(\omega) d\Prob(\omega)
\end{equation*}
The one-variable random attractor $F_\omega$ is then obtained by choosing $\omega\in\Omega$
according to the law $\Prob$.
\subsubsection{One-variable random self-similar sets}
To make useful dimension estimates we have to restrict the class of functions. The simplest model is
that of self-similar sets, where we restrict $f_i^\lambda$ to similarities. That is, $\lvert
f_i^\lambda(x)-f_i^\lambda(y)\rvert = c_i^\lambda \lvert x - y\rvert$ for all $x,y\in V$ and some
$c_i^\lambda >0$.
It is well-known that for self-similar maps and our standing assumptions, the 
Hausdorff and box-counting dimensions are bounded
above by the unique $s$ satisfying
\begin{equation}\label{eq:simdim}
  \E^g \left(\sum_{1\leq i \leq N_{\omega_1}} (c_i^{\omega_1})^s\right) = 1.
\end{equation}
If one further assumes that there exists a non-empty open set $U$ such that the union
$\bigcup_{i=1}^{N_\lambda}f_i^\lambda(U)$ is disjoint for all $\lambda$ and $f_i^\lambda(U)\subseteq
U$ we say that the uniform open set condition holds. Under this assumption the unique $s$ in
\eqref{eq:simdim} coincides with the Hausdorff, box-counting, and quasi-Assouad dimension of
$F_\omega$ for $\Prob$-almost all $\omega$, see e.g.\ \cite{Troscheit17} and references therein.
Since we refer to the sum above quite frequently we write $\fS_\lambda^t = \sum
(c_i^\lambda)^t$.
To avoid the trivial case when the Assouad dimension coincides with the Hausdorff dimension, and
there is nothing to prove as the generalised Assouad dimension coincides with this common
value, we make the assumption that the systems is not almost deterministic. That is,
$\Prob(\fS_{\omega_1}^s=1)\neq 1$, where $s$ is the almost sure Hausdorff dimension.
In particular, this implies that the Assouad dimension is
strictly bigger than the Hausdorff (and upper box-counting) dimension.
In fact, the Assouad dimension of $F_\omega$ is, almost surely, given by 
\begin{equation*}
  \dimA F_\omega = \sup\{s\;:\;
  \mu(\{\lambda\in\Lambda\;:\;\fS_\lambda^s\geq 1\})>0\}
\end{equation*}
see \cite{Fraser14c,Troscheit17}.
To not obscure the result with needless technicality, we only analyse the case when the iterated
function systems are homogeneous, i.e.\ $c_i^\lambda = c_j^\lambda$ for all $\lambda$. We write
$c(\lambda)$ for the common value, then $\fS_\lambda^s = N_\lambda c(\lambda)^s$.
\begin{theorem}\label{thm:SelfSim}
  Let $F_\omega$ be a one-variable random self-similar set generated by homogeneous iterated
  functions systems satisfying the uniform open set condition.
  Then the following dichotomy holds:
  Let $\phi(x)$ be a dimension function such that 
  \begin{equation*}
    \sum_{k=1}^\infty e^{-\phi(e^{-k})k} < \infty.
  \end{equation*}
  Then $\dimAf F_\omega = \dimB F_\omega$ for dimension functions $\fii(x)=\upomega(\phi(x))$, almost surely.

  Conversely, let $\phi(x)$ be a dimension function such that
  \begin{equation*}
    \sum_{k=1}^\infty e^{-\phi(e^{-k})k}=\infty.
  \end{equation*}
  Then $\dimAf F_\omega = \dimA F_\omega$ for dimension functions $\fii(x)=o(\phi(x))$, almost surely.
  Additionally, assume there exists $\Lambda'\subset \Lambda$ such that $\fS_\lambda^{s_A}=1$ for all
  $\lambda\in\Lambda'$, where $s_A$ is the almost
  sure  Assouad dimension of $F_\omega$, and $\mu(\Lambda')>0$. Then the above result holds for all
  $\varphi(x) \leq C(\phi(x))$, where $C>0$ is some constant.
\end{theorem}
We prove this in Section \ref{sect:OneSimProof}.
The methods we have used rely on the individual iterated function systems being homogeneous but we
need not have made this argument. One can construct a subtree such that for this random realisation
the covering numbers are comparable. Since this is somewhat more technical, we leave this case open.

\subsubsection{One-variable random Bedford-McMullen carpets}
Bedford-McMullen carpets are  simple self-affine iterated function systems that are often the
easiest to give as counterexamples to the self-similar theory~\cite{Bedford84,McMullen84}. They consist of non-overlapping
images of the unit square with fixed horizontal and vertical contraction of $1/m$ and $1/n$,
respectively that align in an $m\times n$ grid of the unit square, see Figure~\ref{fig:BMCarpets}
for two examples.
\begin{figure}[htb]
  \includegraphics[width=.45\textwidth]{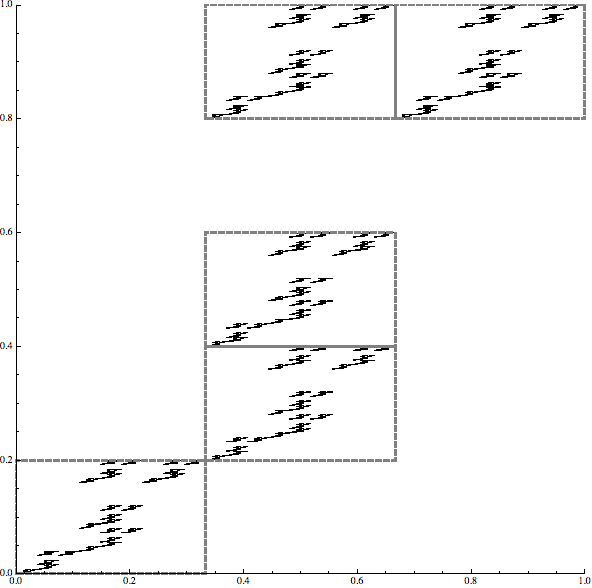}
  \includegraphics[width=.45\textwidth]{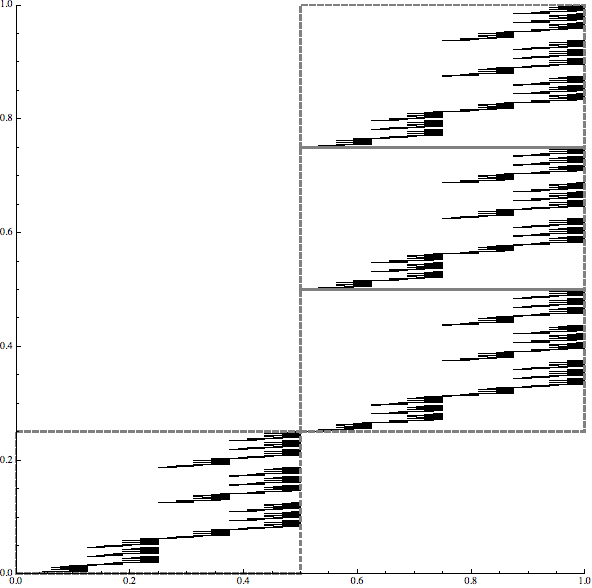}
  \caption{Two Bedford-McMullen carpets with their generating affine maps. The left example $F_1$
  has $m_1=3$, $n_1=5$, $C_1=3$, $B_1=3$, and $N_1=5$. The right example $F_2$ has $m_2=2$, $n_2=4$,
$C_2=3$, $B_2=2$, and $N_2=4$.}
  \label{fig:BMCarpets}
\end{figure}

We randomise the construction in the same one-variable random fashion by choosing different sub rectangles at every
step from the finite collection of possible arrangements. We denote these in the same way and
write $2\leq m_\lambda<n_\lambda \leq N<\infty$ for the subdivisions.
Then, $f_i^\lambda(x)$ are of the form
\[
  f_i^\lambda(x) = \begin{pmatrix} 1/m_\lambda & 0 \\ 0 & 1/n_\lambda \end{pmatrix} x +
  \begin{pmatrix}a_i^\lambda \\ b_i^\lambda \end{pmatrix}
\]
where $a_i^\lambda\in \{0,\dots,m_\lambda-1\}$ and $b_i^\lambda\in\{0,\dots,n_\lambda-1\}$.
The boundedness of $n_\lambda$ implies that there are only finitely many IFSs with finitely many
maps. Hence $\Lambda$ is finite and $\mu$ is finitely supported. We write $p_\lambda=\mu(\lambda)$.

We heavily rely on results in \cite{FraserTroscheit18}, where the $\theta$-Assouad spectrum of these
attractors are found. In fact, this part can be considered an extension of \cite{FraserTroscheit18},
in the sense that the previous work gave a complete characterisation of the spectrum between the
upper box-counting and the quasi-Assouad dimension, whereas we extend this to the Assouad dimension. 
Let $C_\lambda$ be the maximal number of maps that align in a column and $B_\lambda$ be the maximal
number of non-empty columns. Further, let $N_\lambda$ be the number of maps in the IFS indexed by
$\lambda$. See also Figure~\ref{fig:BMCarpets} for an example. Write $\overline{n}=\E^g (n_\lambda)$ and $\overline {m} = \E^g(m_\lambda)$.
Similarly, let $\overline{B},\overline{C}$, and $\overline N$ denote their respective geometric
means.
The almost sure Assouad spectrum was found in \cite{FraserTroscheit18} to be 
\[
  \dimAt F_\omega = 
  \begin{cases}
\dfrac{1}{1- \theta}
\left( \dfrac{ \log \left( \overline{B}\, \overline{C}^\theta \overline{N}^{-\theta}
\right) }{ \log\overline{m}}  +   \dfrac{ \log
  \left(\overline{N}\, \overline{B}^{-1} \overline{C}^{-\theta}\right)}{\sum_{\lambda  } p_\lambda \log n_\lambda}
  \right),  & 0< \theta \leq \dfrac{ \log \overline{m}}{\log \overline{n}}; \\ 
  \dfrac{\log \overline{B} }{\log \overline{m}}  +
  \dfrac{\log \overline{C}}{\log \overline{n}}, &  \dfrac{\log \overline{m}}{ \log \overline{n}}< \theta
 <1\,.
 \end{cases}
\]
From this we can further deduce that
\[
  \dimB F_\omega =   \dfrac{\log  \overline{B} }{
    \log \overline{m}}  +   \dfrac{\log \overline{N}\, \overline{B}^{-1}}{ \log
    \overline n}
\quad\text{ and }\quad
\dimqA = \frac{\log \overline B }{\log\overline m}  +
\frac{\log \overline C}{\log\overline n}.
\]
However, the almost sure Assouad dimension is generally distinct from the quasi-Assouad dimension
and given by
\[
  \dimA F_\omega = \max_{\lambda\in\Lambda} \, \frac{\log B_\lambda}{\log m_\lambda} \, + \, \max_{\lambda\in\Lambda}
  \,  \frac{\log C_\lambda}{\log n_\lambda}.
\]
see \cite{Fraser14c}.
Note, in particular, that the dimension does not depend on the exact form of $\mu$, provided it is
supported on $\Lambda$.  
Our main result in this section is bridging this gap with a similar dichotomy as for self-similar
sets.
\begin{theorem}\label{thm:BedfordMcMullen}
  Let $F_\omega$ be a one-variable random Bedford-McMullen carpet.
  Then the following dichotomy holds:
  Let $\phi(x)\leq \log\overline n / \log \overline m -1$ be a dimension function such that 
  \begin{equation*}
    \sum_{k=1}^\infty e^{-\phi(e^{-k})k} < \infty.
  \end{equation*}
  Then $\dimAf F_\omega = \dimqA F_\omega$ for dimension functions $\fii(x)=\upomega(\phi(x))$, almost surely.

  Conversely, let $\phi(x)$ be a dimension function such that
  \begin{equation*}
    \sum_{k=1}^\infty e^{-\phi(e^{-k})k}=\infty.
  \end{equation*}
  There exists $C>0$ such that $\dimAf F_\omega = \dimA F_\omega$ for all dimension functions
  $\fii(x)\leq C\phi(x)$, almost surely.
\end{theorem}

\begin{remark}
The dichotomies, or phase transitions, observed 
in Theorems \ref{thm:GWupper}, \ref{thm:bottomGW}, \ref{thm:SelfSim} and \ref{thm:BedfordMcMullen}
can be seen as a form of random mass transference principle, as described in \cite[\S 5]{Allen19}.
In the theorems described there, no assumptions are being made on the overlaps and it would be
interesting to know if any separation condition assumptions are needed in our results at all.
\end{remark}

\section{Proofs for Galton-Watson trees}
\label{sect:GWProofs}
In this section we prove Theorems \ref{thm:GWupper} and \ref{thm:bottomGW}. 
We rely on the following lemma.
\begin{lemma}\label{thm:upperGWP}
  Assume there exists $\tau'>0$ such that $\sup_k\E(\exp(\tau' W_k ))<\infty$, then there exists
  $\tau>0$ such that
  \begin{equation*}
  \Prob(Z_k \geq m^{(1+\varepsilon)k}) \leq C \exp(-\tau m^{\varepsilon k}).
\end{equation*}
\end{lemma}
The proof of this lemma follows easily from a standard Chernoff bound, see \cite[Lemma 3.4]{Troscheit19} for
details.
Further, the conditions of Lemma~\ref{thm:upperGWP} are satisfied for supercritical Galton-Watson
processes with finitely
supported offspring distribution, see
Athreya \cite[Theorem 4]{Athreya94} and \cite[\S3]{Troscheit19} for details.

Note that a ball of size $r$ in the metric on $F_\tau$ 
with centre $x\in F_\tau$ is simply the unique subtree containing $x$ that starts at level $k$ satisfying 
$2^{-k}\leq r< 2^{-(k-1)}$. Hence, for two scales $r< R$ the quantity $N_r(B(x,R))$ is equal to
the number of nodes at level $l$ satisfying $2^{-l}\leq r< 2^{-(l-1)}$ that share a common
ancestor with $x$ at level $k$ satisfying $2^{-k}\leq R< 2^{-(k-1)}$.
Using independence, this is an independent copy of $Z_{l-k+1}$ and so 
$N_r(B(x,R))\sim Z_{\log_2(1/r)- \log_2 (1/R)}$.

\begin{proof}[Proof of Theorem \ref{thm:GWupper}]
  Fix $\varepsilon>0$. By Lemma~\ref{thm:upperGWP} we have
  \begin{equation*}
    \Prob\{Z_k \geq m^{(1+\varepsilon)k} \text{ for some }k\geq l\} \leq \sum_{j=l}^\infty C
    \exp(-\tau m ^{\varepsilon k})
    \lesssim  \exp(-\tau m^{\varepsilon l})
  \end{equation*}
  For large enough $k$, there exists $A$ (depending on the realisation) such that $Z_k \leq A m^{k}$, 
  almost surely. Thus, the probability $P_k$ that there is a node at generation k that exceeds
  the average from generation $(1+\phi(e^{-k}))k$ onwards satisfies
  \begin{equation*}
    P_k \lesssim m^{k} \exp(-\tau m^{\varepsilon \phi(e^{-k}) k})
    = \exp\left(k\log m - \tau m^{\varepsilon \phi(e^{-k})k}\right)
    \leq \exp\left(k\log m - \tau m^{(1+\delta)\log k/\log m}\right)
  \end{equation*}
  for some $\delta>0$ be assumption on $\phi$.
  Finally, as $k\log m - \tau k^{(1+\delta)} < - k^{1+\delta/2}$ for large enough $k$, we have
  \begin{align*}
    \sum_{k=1}^\infty P_k 
    \lesssim \sum_{k=1}^\infty \exp\left(k\log m - \tau k^{1+\delta}\right)
    \lesssim \sum_{k=1}^\infty  \exp(-\tau k^{1+\delta/2})<\infty. 
  \end{align*}
  An application of the Borel-Cantelli lemma shows that, almost surely, there exists a level
  $k_0$ from which no node at level $k\geq k_0$ in the Galton-Watson tree will have more than
  $m^{(1+\varepsilon)l}$ 
  many descendants for $l\geq \phi(e^{-k})k$. Geometrically, this
  means that for all $r \sim 2^{-l} < R^{(1+\phi(e^{-k}))k}\sim 2^{-(1+\phi(e^{-k})k}$ small enough we obtain
  \[
    N_r(B(x,R)) \,\sim\, Z_{l-k} \,\lesssim \, m^{(1+\varepsilon)(\log_2(1/r)-\log_2(1/R))} 
    \,=\,\left(\frac{R}{r}\right)^{(1+\varepsilon)\log m/\log 2},
  \]
  which gives the required result.
\end{proof}

\begin{proof}[Proof of Theorem \ref{thm:bottomGW}]
  Let $N=\max\{i\,:\,\theta_i>0\}$ and assume that $N>m$ as otherwise there is nothing to
  prove.
  There exists $q>0$ and $k_0\geq 1$ such that $\Prob(Z_k > m^k)\geq q$ for all $k\geq k_0$ by the
  martingale convergence theorem.
  Write $p = \theta_{N}$. Fix $k$ such that $k-l_k>k_0$, where $l_k=\varepsilon k \log
  m/\log(N/m)$. Consider the
  probability that the maximal branching is chosen in the first $l_k$ levels after the root.
  Then, at level $l_k$,
  there are $N^{l_k}$ descendants. This occurs
  with probability $p \, p^{N} p^{N^2}\cdots p^{N^{l_k-1}}$. Each descendant has
  $m^{k-l_k}$ descendants at level $k$ with probability at least $q$ and therefore the probability
  that $Z_k \geq m^{(1+\varepsilon)k}$ is bounded below by 
  \begin{equation}
    \label{eq:thesecthing}
    p \cdots p^{N^{l_k-1}} q^{N^{l_k}}\geq \varrho^{N^{l_k+1}}
    =\exp(N^{\varepsilon k \log m / \log(N/m)+1}\cdot\log\varrho)
    \end{equation}
  for some $\varrho>0$.
  Let $(k_i)$ be a sequence such that $(1+\phi(e^{-k_i}))k_i < k_{i+1}$.
  Then,
  \begin{equation}
    \widetilde{P}_i =\Prob(\exists\text{ node at level }k_i \text{ s.t. }Z_{n_i}\geq
    m^{(1+\varepsilon)n_i}\text{ for }n_i = \phi(e^{-k_i})k_i 
    \geq \left(1- \left(1-\varrho^{N^{l_{k_i}}}\right)^{Am^k}\right)
    \label{eq:levelprob}
  \end{equation}
by independence and the fact that there are at least $A m^k$ nodes at level $k$.
  Note that combining \eqref{eq:condupper} with \eqref{eq:thesecthing} one obtains
  \[
    \varrho^{N^{l_k+1}}\geq \exp \left( N k^{1-\delta}\log\varrho\right)
  \]
for some $\delta>0$. Since further $A m^k = A \exp(k\log m)$ we obtain
  \begin{multline*}
    \exp\log \left(1-\varrho^{N^{l_{k_i}}}\right)^{Am^k}
    \leq \exp \left(-Am^k \varrho^{N^{l_{k_i}}}\right)\\
    \leq \exp \left( -A \exp\left(k\log m - Nk^{1-\delta} \log(1/\varrho)\right)\right)\to 0
  \end{multline*}
Therefore, for large enough $k$, the quantity in \eqref{eq:levelprob}
is bounded below by $1/2$ and 
  \begin{equation*}
    \sum_i \widetilde{P}_i \;\gtrsim \;\sum_i 1/2 =\infty.
  \end{equation*}
  The disjointness combined with the Borel-Cantelli lemma therefore posit the existence of
  infinitely many $i$ for which such a maximal chain exists. The dimension result directly follows
  by taking $R_i=2^{- k_i}$ and $r_i=2^{-n_i}$ to give a sequence of $N_{r_i}(B(x_i, R_i))$ such that
  \[
    N_{r_i}(B(x_i, R_i)) \gtrsim m^{(1+\varepsilon)(n_i-k_i)} =
    \left(\frac{R}{r}\right)^{(1+\varepsilon) \log m / \log 2}. \qedhere
    \]
\end{proof}

\section{One variable proofs}
\label{sect:OneSimProof}
\subsection{Cram\'er's theorem for i.i.d.\ variables}
Cram\'er's theorem is a fundamental result in large deviations concerning the error of sums of
i.i.d.\ random variables.
Given a sequence of i.i.d.\ random variables $(X_i)_i$, we write $S_n=\sum_{i=1}^n X_i$. The rate function
of this process is defined by the Legendre transform of the moment generating function of the random
variable. That is, the moment generating function is $M(\theta)= \E(\exp(\theta X_1))$ and its
Legendre transform is $I(x) = \sup_{\theta\in\R} \theta x - \log M(\theta)$. 
\begin{theorem}\label{thm:cramer}
	Let $(X_i)_i$ be a sequence of centred i.i.d.\ random variables with common finite moment
	generating function $M(\theta)= \E(\exp(\theta X_1))$.
  Then, if $M(\theta)$ is well-defined for all $\theta$, the following hold:
	\begin{enumerate}[(a)]
		\item For any closed set $F\subseteq \R$,
			\[
				\limsup_{n\to\infty} \frac{1}{n} \log \Prob(S_n\in F) \leq -\inf_{x\in
				F} I(x);
			\]
		\item For any open set $U\subseteq \R$,
			\begin{equation*}
				\liminf_{n\to\infty} \frac{1}{n} \log \Prob(S_n \in U) \geq -\inf_{x\in
				U} I(x).
			\end{equation*}
	\end{enumerate}
\end{theorem}
Letting $F=[a,\infty)$ and $U=(a,\infty)$ for $0<a<\ess\sup X_1$ we have $I(a)>0$ and for all $\delta>0$ there
exists $N_\delta\in\N$ such that
\begin{align*}
  -\inf_{x\in U}I(x) - \delta \leq &\frac{1}{n} \log \Prob\{S_n \in U\} \leq \frac{1}{n} \Prob\{ S_n
  \in F\} \leq - \inf I(x)+\delta\\
  e^{-(I(a)+\delta)n} \leq &\Prob\left\{ \sum_{i=1}^n X_i \geq a n \right\} \leq e^{-(I(a)-\delta)n}
\end{align*}
for all $n\geq N_\delta$ since $I$ is non-decreasing.
Note that this holds for any $n$ large enough and so in particular even if $n$ depends on the stochastic process.

\subsection{A dynamical Borel-Cantelli Lemma}
To establish the strong dichotomy of the almost sure existence of extreme events we will need a
theorem slightly stronger than the second Borel-Cantelli lemma.

Let $E_n$ be a sequence of events such that $\sum \Prob (E_n)
=\infty$. If those events were independent, the second Borel-Cantelli lemma would assert that almost
every $\omega\in\Omega$ is contained in $\omega\in E_n$ for infinitely many $n$, i.e.\
$\Prob(\bigcap_{K=1}^\infty \bigcup_{k=K}^\infty E_k)=1$.
Since we will be dealing with events that are not
independent we will use a stronger version. Define the correlation by 
\[
  \mathcal{R}_{n,m} =
\Prob(E_n \cap E_m)- \Prob(E_n)\Prob(E_m).
\]
The following follows from the work of Sprind\v{z}uk~\cite{Sprindzuk}, see also
\cite[Theorem 1.4]{Chernov01}.
\begin{theorem}
  \label{thm:kleinbock}
  Let $E_n$ be a sequence of events such that $\sum\Prob(E_n)=\infty$. If there exists $C>0$ such
  that
  \[
    \sum_{N\leq n,m\leq M}\mathcal{R}_{n,m} \quad \leq\quad C \sum_{i=N}^M \Prob(E_i)
  \]
  for all $1\leq N < M <\infty$. Then, for $\Prob$-almost every $\omega$, $\omega\in E_n$ for
  infinitely many $n$.
\end{theorem}
\begin{proof}
  This is a direct application of \cite[\S7, Lemma 10]{Sprindzuk} with
  $f_k(\omega)=\chi_{E_k}(\omega)$, $f_k = \varphi_k = \Prob(E_k)$ and the conclusion that $\sum
  \chi_{E_k}(\omega)$ diverges. 
\end{proof}

\subsection{Proof of Theorem~\ref{thm:SelfSim}: self-similar sets}
Let $X_i = \log\fS_{\omega_i}^s = \log
N_{\omega_i}c(\omega_i)^s$. Observe that $\E(X_i) = 0 $ and that $ \ess\sup
\lvert X_\lambda \rvert <\infty$. Hence the moment generating function of $X_i$ is well-defined for
all $\theta$ and we can apply Cram\'er's theorem.
Let $r< R^{1+\fii(R)}<R<R_0$ and set $k(R)$ and $k(r)$ such that 
\begin{equation*}
  \prod_{i=1}^{k(R)}c(\omega_i)\sim R \quad\text{ and }\quad
  \prod_{i=1}^{k(r)}c(\omega_i) \sim r.
\end{equation*}
Since $\phi(e^{-x})x$ is non-increasing we can, without loss of generality,
take $R_0$ (depending on the realisation) small enough such that Cram\'er's theorem holds for
$k(r) - k(R)$.
Thus, for all $\eps>0$,
\begin{equation*}
  \Prob\left\{\sum_{i=k(R)}^{k(r)} X_i\, \geq \,\eps (k(r)-k(R))\right\}\leq
  e^{-(I(\eps)-\delta)(k(r)-k(R))}
\end{equation*}
Therefore, the probability $P(R)$ that there exists $r< R^{1+\fii(R)}<R<R_0$, given $R$ is bounded by
\begin{equation*}
  P(R)\leq\sum_{l=k(R^{1+\fii(R)})}^\infty \Prob\left\{\sum_{i=k(R)}^{l} X_i\, \geq \,\eps (l-k(R))\right\}\leq
  \sum_{l=k(R^{1+\fii(R)})}^\infty e^{-\tau(l-k(R))}\lesssim
  e^{-\tau(k(R^{1+\fii(R)})-k(R))},
\end{equation*}
where we have written $\tau=I(\eps)-\delta$ to ease notation.
Without loss of generation, using Cram\'er's theorem, we can assume that $R_0$ is also chosen small
enough such that
\begin{equation*}
  \E^g(c(\lambda))^{(1+\delta) (k(R^{1+\fii(R)}) - k(R))} \leq
  \prod_{i=k(R)}^{k(R^{1+\fii(R)})}c(\omega_i)\sim \frac{R^{1+\fii(R)}}{R} = R^{\fii(R)}.
\end{equation*}
Thus, 
\begin{equation*}
k(R^{1+\fii(R)}) - k(R)
\geq 
  -\tau'\, \fii\left(\prod_{i=1}^{k(R)}c(\omega_i)\right)\cdot \log\left(\prod_{i=1}^{k(R)}c(\omega_i)\right)
\end{equation*}
for some $\tau'>0$.
Now, for any given $\omega$, the number of levels such that $\log(R)=n$ is uniformly bounded.
Further, the number of products of $\prod c(\omega_i)$ that are comparable to $e^{-n}$ is uniformly
bounded. Therefore the sum over the probabilities that there exists a ball $B(x,R)$ at level
$k(R)$ such that $X_i$ exceeds the mean by more than $\eps$ is bounded by
\[
  \sum_{\log R = 1}^\infty P(R) \leq C \sum_{n=1} e^{-\tau\tau' \fii(R) \log 1/R} \leq C' \sum_{n=1}
  e^{-\tau\tau' \fii(e^{-n})n}<C'' \sum_{n=1}^\infty e^{-\phi(e^{-n})n}<\infty.
\]
By the Borel-Cantelli lemma this happens only finitely many times, almost surely.
Finally, we can conclude that almost surely for small enough $R$ (depending on the realisation)
there are no pairs
$r<R^{1+\fii(R)}$ such that $\sum_{i=k(R)}^{k(r)}X_i > \eps (k(r)-k(R))$. Then
\[
  \sum_{i=k(R)}^{k(r)}\log N_{\omega_i}c(\omega_i)^s \leq \eps(k(r)-k(R)).
  \]
Observe that the number of $r$ coverings is comparable to the number of descendants of the $B(x,R)$
cylinder. Therefore,
\[
  N_r(B(x,R)\cap F_\omega) \sim \prod_{i=k(R)}^{k(r)}N_{\omega_i} \lesssim
  \prod_{i=k(R)}^{k(r)}
  \frac{e^{\eps}}{c(\omega_i)^s} \lesssim \left( \frac{R}{r}\right)^{s+\eps'} 
\]
for some $\eps'$ such that $\eps'\to 0$ as $\eps\to 0$. Since $\eps$ was arbitrary, we have the
desired conclusion for the first part.

\vskip.5em

We now prove the second half of the theorem. 
Recall that the almost sure Assouad dimension $s_A$ of $F_\omega$ is given by $s_A =
\sup_{\lambda\in\supp\mu}\{- \log N_\lambda / \log c(\lambda)\}$.
Let $\eps>0$ and take 
\[
  T_{\eps}=\left\{ \lambda\in\Lambda\;:\; \frac{-\log N_\lambda }{ \log c(\lambda)}\geq
    s_A-\eps
\right\}\text{ and }p_{\eps}=\mu(T_{\eps}).
\]
Define $c_{\sup} = \sup_{\lambda} c(\lambda)$ and $c_{\inf}=\inf_{\lambda} c(\lambda)$.
Let $\psi(n)=\fii((c_{\sup})^{n})\gamma$, where $\fii$ is given as $\fii(R)=o(\phi(R))$ and
$\gamma=\log c_{\inf}/\log c_{\sup}$. Recall that $\psi$ is 
non-increasing and consider the events 
\[E_n=\{\omega\in\Omega \;:\;\omega_i\in
T_{\eps}\text{ for }n\leq i<n+\psi(n)n\}.\]
Clearly, $\Prob(E_n) = p_{\eps}^{\psi(n)n}$. The event $E_n\cap E_m$ for $n\leq m$ has probability
\begin{equation*}
  \Prob(E_n\cap E_m) = p_{\eps}^{\psi(m)m + m -n}
\end{equation*}
due to the overlap of $[n,\psi(n)n]\cap [m,\psi(m)m]$. The correlation is 
\begin{equation*}
  \mathcal{R}_{n,m} =  \Prob(E_n\cap E_m)  -  \Prob(E_n)\Prob(E_m)
  = p_{\eps}^{\psi(m)+m-n} - p_{\eps}^{\psi(n)n+\psi(m)m}
\end{equation*}
Therefore
\begin{multline*}
  \sum_{N\leq n,m M} \mathcal{R}_{n,m}
  \leq
  2\sum_{m=N}^M \sum_{n=N}^m \mathcal{R}_{n,m}
  \leq 2 \sum_{m=N}^M p_{\eps}^{\psi(m)m}\sum_{n=N}^m p_{\eps}^{m-n}\\\leq 2 \sum_{m=N}^M
  p_{\eps}^{\psi(m)m}\sum_{i=0}^\infty p_{\eps}^i
  \leq C_{\eps} \sum_{m=N}^M p_{\eps}^{\psi(m)m} \leq C_{\eps}\sum_{m=N}^M \Prob(E_m).
\end{multline*}
for all $1\leq N < M <\infty$. To use Theorem~\ref{thm:kleinbock} it remains to check divergence the
latter sum. As we are in the diverging case, $\fii(x)x$ increases as $x$ tends to $0$ and $\fii(x) =
o(\phi(x))$ as $x\to0$. Then,
\begin{align}
  \sum_{n=1}^\infty \Prob(E_m) = \sum_{n=1}^\infty p_{\eps}^{\psi(n)n}
  &= \sum_{n=1}^\infty e^{-\fii((c_{\sup})^n)\gamma n\log(1/p_{\eps})}\nonumber\\
  &=\sum_{n=1}^\infty
  \exp\left(-\fii(e^{-n\log(1/c_{\sup})})n\log(1/c_{\sup})\gamma\frac{\log(1/p_{\eps})}{\log(1/c_{\sup})}\right)\label{eq:fraction}\\
  &\gtrsim \sum_{n=1}^\infty
  \exp\left(-\phi(e^{-n\log(1/c_{\sup})})n\log(1/c_{\sup})\right)\label{eq:littleo}\\
  &\sim\sum_{n=1}^\infty \exp\left(-\phi(e^{-n})n\right) = \infty\label{eq:substitution}
\end{align}
where we have used the integration test and the substitution rule to obtain \eqref{eq:substitution}.
We have obtained \eqref{eq:littleo} by $\fii(x)=o(\phi(x))$ to combat the final
fraction in \eqref{eq:fraction}. However, if $p_{\eps}$ is bounded away from $0$ as $\eps\to 0$ we
can sharpen this to $\fii(x) \leq C\phi(x)$ by taking $\eps=0$ and using the bound on $p_{\eps}$.
This can happen when there exists $\Lambda'\subset\Lambda$ with $\mu(\Lambda')>0$ that maximises
$-\log N_\lambda/\log c(\lambda)$, i.e.\ when $\fS_\lambda^{s_A}=1$ for all $\lambda\in\Lambda'$.

Application of Theorem~\ref{thm:kleinbock} gives us that $\omega \in E_n$ for infinitely
many $n$, almost surely.
That is, given a generic $\omega\in\Omega$, there are infinitely many $n$ such that $\omega_k \in
T_{\eps}$ for $n\leq k \leq \psi(n)n$. Therefore, considering the ball
$f_{\omega\rvert_n}(\Delta)\cap F_\omega$ of diameter $R\sim \prod_{k=1}^{n}c(\omega_k)$, we can use
the fact that the interiors are separated and standard arguments (see e.g.\ \cite[Lemma 3.2]{Troscheit19}) to claim that this
ball must be covered by at least
$C\prod_{k=n}^{(1+\psi(n))n}N_{\omega_k}$ many balls of radius
$r\sim\prod_{k=1}^{(1+\psi(n))n}c(\omega_k)$.
Therefore, there exist $x,r,R$ such that
\begin{equation*}
  N_r(B(x,R)\cap F_\omega) \gtrsim
  \prod_{k=n}^{(1+\psi(n))n}N_{\omega_k}\geq \prod_{k=n}^{(1+\psi(n))n} c(\omega_k)^{-(s-\eps)} \sim
  \left( \frac{R}{r} \right)^{s_A - \eps}.
\end{equation*}
Finally, we check that $r \lesssim R^{1+\fii(R)}$. Equivalently we check whether $r/R \lesssim R^{\fii(R)}$.
\begin{equation*}
  \frac{r}{R} \sim \prod_{k=n}^{(1+\psi(n))n}c(\omega_k)\leq (c_{\sup})^{\psi(n)n}
  =\exp\left(\fii((c_{\sup})^n) n \gamma \log(1/c_{\sup})\right)
  =(c_{\inf})^{\fii((c_{\sup})^n) n}\leq R^{\fii(R)}
\end{equation*}
as required.
Therefore $\dimAf F_\omega \geq s_A - \eps$ almost surely.
Since $\eps>0$ was arbitrary (or can in cases be chosen to be $0$) we obtain the required result.

\subsection{Proof of Theorem~\ref{thm:BedfordMcMullen} Bedford-McMullen carpets}
\begin{proof}
  We define the random variables $k_1^\omega(R), k_2^\omega(R)$ as the levels when the rectangles in
  the construction have base length $R$ and height $R$, respectively. That is,
  \[
    \prod_{i=1}^{k_2^\omega(R)}n^{-1}_{\omega_i} \sim R \quad\text{and}\quad
    \prod_{i=1}^{k_1^\omega(R)}m^{-1}_{\omega_i} \sim R.
  \]
  It follows from the estimates in \cite{FraserTroscheit18} that 
  \begin{equation}\label{eq:JMFSTestimate}
    N_{R^{1+\fii(R)}}(B(x,R)\cap F_\omega) \sim
    \prod_{l=k_2^\omega(R)}^{k_2^\omega(R^{1+\fii(R)})} C_{\omega_l} 
    \prod_{l=k_1^\omega(R)}^{k_1^\omega(R^{1+\fii(R)})} B_{\omega_l}.
  \end{equation}
  Let $X_i = \log C_{\omega_i} - \log\Cbar$ and $Y_i = \log B_{\omega_i} - \log\Bbar$, where
  $\Cbar=\E^g(C_\lambda)$ and $\Bbar = \E^g (C_\lambda)$.
  As in the self-similar case we have $E(X_i) = E(Y_i) =0$ and due to the finiteness of $\Lambda$,
  the moment generating function exists for all $\theta$. Hence we can apply Cram\'er's theorem. Let
  $r < R^{1+\fii(R)} < R < R_0$, where $R_0$ is chosen small enough such that Cram\'er's theorem
  holds for $k_i^\omega(R) - k_i^\omega(r)$, ($i=1,2$).
  Thus, for all $\eps>0$,
\begin{align*}
 & \Prob\left\{\sum_{i=k_2^\omega(R)}^{k_2^\omega(r)} X_i+\sum_{i=k_1^\omega(R)}^{k_1^\omega(r)}
Y_i\, \geq \,\eps (k_2^\omega(r)-k_2^\omega(R)+k_1^\omega(r)-k_1^\omega(R))\right\}\\&\leq
\Prob\left\{\sum_{i=k_2^\omega(R)}^{k_2^\omega(r)} X_i \geq \,\eps (k_2^\omega(r)-k_2^\omega(R))\right\}
+\Prob\left\{\sum_{i=k_1^\omega(R)}^{k_1^\omega(r)}
Y_i\, \geq \,\eps (k_1^\omega(r)-k_1^\omega(R))\right\}.
\end{align*}
Applying Cram\'er's theorem to both probabilities and analogous to the self-similar case, we obtain
that, almost surely, for all $r<R^{1+\fii(R)}$ with $R$ small enough (and depending on the realisation)
that 
\begin{equation*}
\sum_{i=k_2^\omega(R)}^{k_2^\omega(r)} X_i+\sum_{i=k_1^\omega(R)}^{k_1^\omega(r)}
Y_i\, < \,\eps (k_2^\omega(r)-k_2^\omega(R)+k_1^\omega(r)-k_1^\omega(R)).
\end{equation*}
  Thus following the same argument as in \cite{FraserTroscheit18}, and using
  \eqref{eq:JMFSTestimate}, gives
\begin{equation*}
 N_{R^{1+\fii(R)}}(B(x,R)\cap F_\omega) \lesssim
  \Cbar^{(1+\eps')(k_2^\omega(R^{1+\phi(R)})-k_2^\omega(R))}
\Bbar^{(1+\eps')(k_1^\omega(R^{1+\phi(R)})-k_1^\omega(R))}
\lesssim \left( \frac{R}{r} \right)^{(1+\eps'')s_q},
\end{equation*}
where $s_q=\ess\dimqA F_\omega$ and $\eps''\to0$ as $\eps\to 0$. 
As $\eps>0$ was arbitrary we conclude that $\dimAf F_\omega = \dimqA F_\omega$ almost surely.
This concludes the first part.

\vskip.5em
For the second part, recall that the almost sure Assouad dimension $s_A$ of $F_\omega$ is given by 
\[
  s_A = \max_{\lambda\in\Lambda}\frac{\log B_\lambda}{\log 
  m_\lambda}+\max_{\lambda\in\Lambda}\frac{\log C_{\lambda}}{\log n_\lambda}.
\]
Without loss of generality assume the first summand is maximised by $1\in\Lambda$, whereas the
  second is maximised by $2\in\Lambda$ (where we may identify $1\sim 2$ if necessary).
Define $\psi(j) = \fii(n_{\min}^j)\log n_2 / \log n_{\max}$, where $\fii(x) \leq \gamma \phi(x)$
with
\[
  \gamma = \frac{\log n_{\min}\log n_{\max}}{(\log(1/p_2)+\kappa
\log(1/p_1))\log n_{2}}\quad\text{and}\quad\kappa = \log n_2 / \log m_1 . 
\]
  Consider the events
  \begin{multline*}
  E_l = \{ \omega\in \Omega \;:\; \omega_i = 2 \text{ for }l\leq i < \psi(l)l
  \text{ and }\omega_i = 1 \text{ for } l'\leq i < l'+ \kappa \psi(l)l
  \\\text{ where } l'=k_1^\omega(R) \text{ and }R \text{ is the least value satisfying
  }k_2^\omega(R)=l\}.
\end{multline*}
  Almost surely, $k_1^\omega(R) - k_2^\omega(R) \gg \psi(l)l$ for large enough $l$. Therefore
  $E_l$ consists of two fixed strings of letters $2$ and $1$ of lengths $\psi(l)l$ and
$\kappa\psi(l)l$, respectively, that do not overlap.
This further gives $\Prob(E_l) \sim p_2^{\psi(l)l}p_1^{\kappa\psi(l)l}$.
Considering an $\omega\in E_{l}\cap E_j$ for $j\leq l$ there are only two cases that appear (with large
probability):
\begin{itemize}
  \item $[j,j+\psi(j)j]$ intersects $[l,l+\psi(l)l]$ and $[j',j'+\kappa\psi(j)j]$ intersects
    $[l',l'+\kappa\psi(l)l]$.
  \item $[j',j'+\psi(j)j]$ intersects $[l,l+\psi(l)l]$.
\end{itemize}
In the first case the probability is given by 
\begin{equation}\label{eq:probEstimate}
  \Prob(E_j\cap E_l) \sim p_2^{\psi(l)l+l-j} p_1^{\kappa\psi(l)l+l'-j'} \leq
  p_2^{\psi(l)l+l-j}p_1^{\kappa\psi(l)l+\tau(l-j)}\sim \Prob(E_l) p_2^{l-j}p_1^{\tau(l-j)}.
\end{equation}
The estimate that $l-j\lesssim l'-j'$ arises from the observation that the $R_l$ associated with $l$ is related to $R_j$
by $R_l/R_j \leq n_{\min}^{-(l-j)}$ and $R_l/R_j \geq m_{\max}^{-(l'-j')}$. This gives $\tau\geq
\log n_{\min}/\log m_{\max}$.
Note further that the last term in \eqref{eq:probEstimate} implies uniform summability over $1\leq j \leq l$.

The second case can only occur if the maximal letters are identical, that is $1=2$ and $p_1 = p_2$.
This gives 
\[
  \Prob(E_j\cap E_l) \sim p_2^{\kappa\psi(l)l+\psi(l)l+l-j'+\psi(j)j} \sim \Prob(E_l)
  p_2^{l-j'+\psi(j)j},
\]
which is also uniformly summable over $1\leq j \leq l$.

We can conclude that for any $1\leq J\leq l$,
\[
  \sum_{j=J}^{l}\mathcal{R}_{j,l} \leq \sum_{j=J}^l\Prob(E_j\cap E_l) \leq \sum_{j=1}^l \Prob(E_j \cap
  E_l) \lesssim  \Prob(E_l)
\]
and so for $1\leq J\leq L$,
\[
  \sum_{J\leq j,l \leq L}\mathcal{R}_{j,l} \leq 2\sum_{l=J}^L \sum_{j=J}^l\Prob(E_j\cap E_l)
  \lesssim 
  \sum_{l=J}^L \Prob(E_l).
\]

To use Theorem~\ref{thm:kleinbock} we need to show that $\sum \Prob(E_j) = \infty$.
This is similar to proving divergence in Theorem~\ref{thm:SelfSim}:
\begin{align*}
  \sum_{j=1}^\infty \Prob(E_j) \sim \sum_{j=1}^{\infty} p_2^{\psi(l)l}p_1^{\kappa\psi(l)l}
  &=\sum_{j=1}^{\infty} \exp(-(\log(1/p_2)+\kappa \log(1/p_1))\psi(l)l) \\
  &=\sum_{j=1}^{\infty} \exp(-\phi((n_{\min})^{-j})j \log n_{\min} )\\
  &\sim \sum_{j=1}^{\infty} \exp(-\phi(e^{-j})j )
  =\infty.
\end{align*}
Hence the conditions of Theorem~\ref{thm:kleinbock} are satisfied and $E_l$ happens infinitely
often.

As $\omega\in E_j$ for infinitely many $j$, almost surely, we have $k_2^\omega(R)=j$ and
$k_1^\omega(R)=j'>(1+\psi(j)j$ for arbitrarily large $j$. Then, by the definition of $k_2$,
and $\psi$,
we have $R^{1+\fii(R)} \geq R n_2^{-\psi(j)j}$ and $R^{\fii(R)} \geq n_2^{-\psi(j)j}$. Similarly, 
$R^{\fii(R)} \geq m_1^{-\kappa\psi(j)j}= n_2^{-\psi(j)j}$ and thus by the estimate
\eqref{eq:JMFSTestimate}, 
\begin{align*}
  N_{R^{1+\fii(R)}}(B(x,R)\cap F_\omega) &\sim
  \prod_{l=k_2^\omega(R)}^{k_2^\omega(R^{1+\fii(R)})} C_{\omega_l} 
  \prod_{l=k_1^\omega(R)}^{k_1^\omega(R^{1+\fii(R)})} B_{\omega_l}\\
  &=C_2^{k_2^\omega(R^{1+\fii(R)})-k_2^\omega(R)}B_1^{k_1^\omega(R^{1+\fii(R)})-k_1^\omega(R)}\\
  & = R^{-\fii(R)(\log C_2/\log n_2+\log B_1/\log m_1)}
\end{align*}
for infinitely many pairs $(x_i, R_i)$ almost surely.
Therefore the almost sure generalised Assouad spectrum with respect to $\fii$ is equal to the
almost sure Assouad dimension.
\end{proof}

\subsection*{Acknowledgements}
Part of this work was started when the author visited Acadia University in October 2018. ST
thanks everyone at Acadia and Wolfville, Nova Scotia for the pleasant stay. 
The author further thanks Jayadev Athreya, Jonathan M.~Fraser, Kathryn Hare, Franklin Mendivil, and
Mike Todd for many fruitful discussions.
The author does not thank that one driver in Wolfville who nearly
killed him at a pedestrian crossing and then swore at him for having the right of way.

\end{document}